\documentclass{amsart}
\usepackage{graphicx}
\usepackage{url}
\usepackage{amssymb}
\usepackage{hyperref}
\usepackage{amssymb}
\usepackage{url}
\usepackage{cite}

\newtheorem{theorem}{Theorem}[section]
\theoremstyle{plain}

\newtheorem{corollary}{Corollary}[section]

\newtheorem{lemma}{Lemma}[section]

\newtheorem{problem}{Problem}
\newtheorem{proposition}{Proposition}[section]
\newtheorem{observation}{Observation}

\numberwithin{equation}{section}
\begin{document}
\title[Ramsey theory and strength of graphs]{Ramsey theory and strength of
graphs}
\author{Rikio Ichishima}
\address{Department of Sport and Physical Education, Faculty of Physical
Education, Kokushikan University, 7-3-1 Nagayama, Tama-shi, Tokyo 206-8515,
Japan}
\email{ichishim@kokushikan.ac.jp}
\author{Francesc A Muntaner-Batle}
\address{Graph Theory and Applications Research Group, School of Electrical
Engineering and Computer Science, Faculty of Engineering and Built
Environment, The University of Newcastle, NSW 2308 Australia }
\email{famb1es@yahoo.es}
\author{Yukio Takahashi}
\address{Department of Science and Engineering, Faculty of Electronics and
Informations, Kokushikan University, 4-28-1 Setagaya, Setagaya-ku, Tokyo
154-8515, Japan}
\email{takayu@kokushikan.ac.jp}

\begin{abstract}
A numbering $f$ of a graph $G$ of order $n$ is a labeling that assigns
distinct elements of the set $\left\{ 1,2,\ldots ,n\right\} $ to the
vertices of $G$, where each $uv\in E\left( G\right) $ is labeled $f\left(
u\right) +f\left( v\right) $. The strength $\mathrm{str}\left( G\right) $ of 
$G$ is defined by 
$\mathrm{str}\left( G\right) =\min \left\{ \mathrm{str}_{f}\left( G\right)
\left\vert f\text{ is a numbering of }G\right. \right\}$,
where $\mathrm{str}_{f}\left( G\right) =\max \left\{ f\left( u\right)
+f\left( v\right) \left\vert uv\in E\left( G\right) \right. \right\} $. 
Let $f\left( n\right) $ denote the maximum of $\mathrm{str}\left( G\right) +%
\mathrm{str}\left( \overline{G}\right) $ over nonempty graphs $G$ and $%
\overline{G}$ of order $n$, where $\overline{G}$ represents the
complement of $G$. 
In this paper, we establish lower bounds for the Ramsey numbers related to the concept of strength of a graph and show a sharp lower bound for $f\left( n\right) $. 
In addition to these results, we provide another lower bound for $f\left( n\right) $ and determine some exact values for $f\left(n\right)$.
Furthermore, we extend existing necessary and sufficient conditions involving the strength of a graph. Finally, we investigate bounds for 
$\mathrm{str}\left( G\right) +\mathrm{str}\left( \overline{G}\right) $ whenever $G$ and $\overline{G}$ are nonempty graphs of order $n$. The resulting upper bound is shown to be related to Ramsey numbers. Throughout this paper, we propose some open problems arising from our study.
\end{abstract}

\date{Sept 12, 2024}
\dedicatory{In memory of Susana Clara L\'{o}pez Masip}
\maketitle

\section{Introduction}
We refer to the book by Chartrand and Lesniak \cite{CL} for
graph-theoretical notation and terminology not described in this paper. 
The \emph{vertex set} of a graph $G$ is denoted by $V \left(G\right)$, 
while the \emph{edge set} of $G$ is denoted by $E\left (G\right)$. 
The cardinality of the vertex set of a graph $G$ is called the \emph{order} of $G$, while the cardinality of the
edge set is the \emph{size} of $G$.
The \emph{complete graph} and \emph{path} of order $n$ are denoted by $K_{n}$ and $P_{n}$, respectively. The \emph{complete bipartite graph} of order $m+n$ and size $mn$ is denoted by $K_{m,n}$. 

If $u$ and $v$ are not adjacent in $G$, 
the \emph{addition of edge} $uv$ results in the smallest supergraph of $G$ containing the edge $uv$ and is denoted by $G+uv$.
For two graphs with disjoint vertex sets, the \emph{union} $G = G_{1}\cup G_{2}$ has $V\left(G\right)=V\left(G_{1}\right) \cup V\left(G_{2}\right)$ 
and $E\left(G\right)=E\left(G_{1}\right) \cup E\left(G_{2}\right)$. 
If a graph $G$ consists of $k$ ($k \geq 2$) disjoint copies of a graph $H$, then we write $G=kH$. 

We will use the notation $\left[a, b\right] $ for the interval of integers $x $ such that $a\leq x\leq b$. For a graph $G$ of order $n$, a \emph{%
numbering} $f$ of $G$ is a labeling that assigns distinct elements of the
set $\left[ 1,n\right] $ to the vertices of $G$, where each $uv\in E\left(
G\right) $ is labeled $f\left( u\right) +f\left( v\right) $. The \emph{%
strength} \textrm{str}$_{f}\left( G\right) $ \emph{of a numbering} $%
f:V\left( G\right) \rightarrow \left[ 1,n\right] $ of $G$ is defined by
\begin{equation*}
\mathrm{str}_{f}\left( G\right) =\max \left\{ f\left( u\right) +f\left(
v\right) \left| uv\in E\left( G\right) \right. \right\}\text{,}
\end{equation*}%
that is, $\mathrm{str}_{f}\left( G\right) $ is the maximum edge label of $G$
and the \emph{strength\ }\textrm{str}$\left( G\right) $ of a graph $G$
itself is 
\begin{equation*}
\mathrm{str}\left( G\right) =\min \left\{ \mathrm{str}_{f}\left( G\right)
\left| f\text{ is a numbering of }G\right. \right\}\text{.}
\end{equation*}
Since empty graphs do not have edges, this definition does not apply to such
graphs. Consequently, we only consider nonempty graphs in our study of the strength. This concept was proposed by Gary Chartrand \cite{Chartrand} and was studied in \cite{IMO1} as a generalization of the problem of finding whether a graph is super edge-magic or not (see \cite{ELNR} for the definition of a super edge-magic graph and also consult either \cite{AH} or \cite{FIM} for alternative and often more useful definitions of the same concept). A necessary and sufficient condition for a graph to be super edge-magic established in \cite{FIM} gives rise to the concept of consecutive strength labeling of a graph (see \cite{IMO1} for the definition of a consecutive strength labeling of a graph), which is equivalent to super edge-magic labeling. It is important to notice that
the term “strength” was introduced by Cunningham \cite{Cunningham} for a different concept in 1985.

Other related parameters have been studied in the area of
graph labeling. Excellent sources for more information on this topic are
found in the dynamic survey by Gallian \cite{Gallian}, which also includes
information on other kinds of graph labeling problems and their applications.

A graph $G$ has no isolated vertices if and only if $\delta \left( G\right)
\geq 1$, where $\delta \left( G\right) $ denotes the minimum degree of $G$.
This gives us the following observation.

\begin{observation}
\label{observation} If $G$ is a graph with $\delta \left( G\right) \geq 1$,
then $\mathrm{str}\left( G\cup nK_{1}\right) =\mathrm{str}\left( G\right) $
for every positive integer $n$.
\end{observation}

In light of Observation \ref{observation}, we only need to be concerned about determining the strength of graphs $G$ with $\delta \left( G\right) \geq 1$. 
For such graphs $G$, the following lower bound for $\mathrm{str}\left( G\right) $ was
found in \cite{IMO1}.

\begin{lemma}
\label{lemma_trivial}For every graph $G$ of order $n$ with $\delta \left(
G\right) \geq 1$,%
\begin{equation*}
\mathrm{str}\left( G\right) \geq n+\delta \left( G\right)\text{.}
\end{equation*}
\end{lemma}

The lower bound given in Lemma \ref%
{lemma_trivial} is sharp in the sense that there are infinitely many graphs $%
G$ for which $\mathrm{str}\left( G\right) =\left| V\left( G\right) \right|
+\delta \left( G\right) $ (see \cite{GLS, IMO1, IMO2, IMOT, IMT2} for a detailed list
of graphs that meet this bound). 
Several other bounds for the strength of a graph have been found in terms of other parameters defined on graphs (see \cite{GLS, IMO1, IMO4, IMOT, IMT, IMT2}). 

For every nonempty graph $G$ of order $n$, it is clear that $3 \leq \mathrm{str}\left( G\right) \leq 2n-1$. 
It is also true that for every $k\in \left[ 1,n-1\right] $, 
there exists a graph $G$ of order $n$ satisfying $\delta \left( G\right) =k$ and 
$\mathrm{str}\left( G\right) =n+k$ (see \cite{IMO3}).

In the process of settling the problem (proposed in \cite{IMO1}) of finding
sufficient conditions for a graph $G$ of order $n$ with $\delta \left(
G\right) \geq 1$ to ensure that $\mathrm{str}\left( G\right) =n+\delta
\left( G\right) $, the following class of graphs was defined in \cite{IMO4}.
For integers $k\geq 2$, let $F_{k}$ be the graph with $V\left( F_{k}\right)
=\left\{ v_{i}\left\vert i\in \left[ 1,k\right] \right. \right\} $ and 
\begin{equation*}
E\left( F_{k}\right) =\left\{ v_{i}v_{j}\left\vert i\in \left[
1,\left\lfloor k/2\right\rfloor \right] \text{ and }j\in \left[ 1+i,k+1-i%
\right] \right. \right\}\text{.}
\end{equation*}
Hence, we have $F_{2}=K_{2}$, $F_{3}=P_{3}$, and $F_{4}=K_{1,3}+e$.

Let $\overline{G}$ denote the complement of a graph $G$. The
following result found in \cite{IMO4} provides a necessary and sufficient
condition for a graph $G$ of order $n$ to hold the inequality $\mathrm{str}%
\left( G\right) \leq 2n-k$, where $k\in \left[ 2,n-1\right] $.

\begin{theorem}
\label{main1} Let $G$ be a graph of order $n$. Then $\mathrm{str}\left(
G\right) \leq 2n-k$ if and only if $\overline{G}$ contains $F_{k}$ as a
subgraph, where $k\in \left[ 2,n-1\right] $.
\end{theorem}

The following theorem that is the contrapositive of Theorem \ref{main1} provides a necessary and sufficient condition
for a graph $G$ of order $n$ to hold the inequality $\mathrm{str}\left(
G\right) \geq 2n-k+1$, where $k\in \left[ 2,n-1\right] $.

\begin{theorem}
\label{main2}Let $G$ be a graph of order $n$. Then $%
\mathrm{str}\left( G\right) \geq 2n-k+1$ if and only if $\overline{G}$ does
not contain $F_{k}$ as a subgraph, where $k\in \left[ 2,n-1\right] $.
\end{theorem}

The preceding two results play an important role in the study of the strength of graphs (see \cite{IMT, IOT, IOT2}).

The following theorem is a consequence of Lemma \ref{lemma_trivial} and Theorem %
\ref{main1}.

\begin{theorem}
\label{main3}Let $G$ be a graph of order $n$ with $\delta \left( G\right)
=n-k$, where $k\in \left[ 2,n-1\right] $. Then $\mathrm{str}\left( G\right)
=2n-k$ if and only if $\overline{G}$ contains $F_{k}$ as a subgraph.
\end{theorem}

For positive integers $s$ and $t$, the (classical) \emph{Ramsey number} $%
r\left( s,t\right) $ is the least positive integer $n$ such that for every
graph $G$ of order $n$, either $G$ contains $K_{s}$ as a subgraph or $%
\overline{G}$ contains $K_{t}$ as a subgraph, that is, $G$ contains either $%
s $ mutually adjacent vertices or $t$ mutually nonadjacent vertices.

The subject of Ramsey numbers has expanded greatly and in many directions
during the past four decades, in many types of \textquotedblleft Ramsey
numbers\textquotedblright (see the paper by Chartrand and Zhang \cite{CP}). In this paper, we consider the following
generalization. For two graphs $G_{1}$ and $G_{2}$, the (generalized) \emph{%
Ramsey number} $r\left( G_{1},G_{2}\right) $ is the least positive integer $%
n $ such that for every graph $G$ of order $n$, either $G$ contains $G_{1}$
as a subgraph or $\overline{G}$ contains $G_{2}$ as a subgraph. Hence, $%
r\left( K_{s},K_{t}\right) =r\left( s,t\right) $. Since $\overline{\overline{%
G}}=G$ for every graph $G$, it follows that $r\left(G_{1},G_{2}\right) =r\left( G_{2},G_{1}\right)$. 
It is also convenient to note that $H_{1}\subseteq G_{1}$ and  $H_{2}\subseteq G_{2}$ imply $r\left(H_{1},H_{2}\right) \leq r\left(G_{1},G_{2}\right)$, and if $G_{1}$ has order $s$ and $G_{2}$ has order $t$, then $r\left(G_{1},G_{2}\right) \leq r\left(s,t\right)$. In particular, we have $r\left(F_{s}, F_{t}\right) \leq r\left(s,t\right)$. For any two positive integers $s$ and $t$, an upper bound for the Ramsey numbers $r\left(s, t\right)$ was found by Erd\"{o}s and Szekeres \cite{ES}, and the same bound was later rediscovered by Greenwood and Gleason \cite{GG}. Therefore, all Ramsey numbers $r\left(s,t\right)$ exist, and so do all Ramsey numbers $r\left(F_{s},F_{t}\right)$.

For an excellent introduction to the theory of Ramsey numbers, the reader is referred to the book by Graham, Rothschild, Spencer, and Solymosi \cite{Graham-Rothschild-Spencer-Solymosi}, and the dynamic survey “Small Ramsey numbers” by Radziszowsiki \cite{Radziszowsiki} provides a wealth of information on such Ramsey numbers. There are many interesting applications of Ramsey theory, these include results in theoretical computer science. The dynamic survey “Ramsey Theory Applications” by Rosta \cite{Rosta} is an excellent source for such applications.

In this paper, we present some results on the Ramsey number $r\left(
F_{s},F_{t}\right)$ and then apply them to establish lower bounds for $\max\left\{\mathrm{str}\left( G\right) +\mathrm{str}\left( \overline{G}\right)\right\} $, where the maximum is taken over all graphs $G$ of order $n$ such that neither $G$ nor $\overline{G}$ is an empty graph. 
We also extend existing necessary and sufficient conditions involving the strength of a graph.
Moreover, we investigate bounds for $\mathrm{str}\left( G\right) +\mathrm{str}\left( \overline{G}\right) $ whenever $G$ and $\overline{G}$ are nonempty graphs of order $n$.
The resulting upper bound is shown to be related to Ramsey numbers.
This paper proposes some open problems arising from our study.

\section{A Lower bound and some exact values for $r\left(F_{s},F_{t}\right)$}

In this section, we establish a lower bound for the Ramsey number $r\left(F_{s}, F_{t}\right) $ and determine some exact values for $r\left( F_{s}, F_{t}\right)$. 

To present our first result, we will utilize the next notable theorem of Chv\'{a}tal \cite%
{Chvatal}.

\begin{theorem}
\label{Chavatal}Let $T_{s}$ be any tree of order $s\geq 2$. For every integer $t \geq 2$, 
\begin{equation*}
r\left( T_{s},K_{t}\right) =1+\left( s-1\right) \left( t-1\right)\text{.}
\end{equation*}
\end{theorem}

We now provide a lower bound for $r\left( F_{s},F_{t}\right) $.

\begin{theorem}
\label{ramsey_lower} For every two integers $s$ and $t$ with $2\leq s\leq t$, 
\begin{equation*}
r\left( F_{s},F_{t}\right) \geq 1+\left( s-1\right) \left\lfloor
t/2\right\rfloor\text{.}
\end{equation*}
\end{theorem}

\begin{proof}
First, recall the definition of $F_{s}$. Then $v_{1}$ is adjacent to $v_{j}$
($j\in \left[ 2,s\right] $), so $K_{1,s-1}\subseteq F_{s}$. 
Next, let $W=\left\{ w_{i}\left| i\in \left[ 1,\left\lfloor t/2\right\rfloor +1\right]
\right. \right\} $. Then the subgraph $H$ of $F_{t}$ induced by $W$ has the
vertex set $V\left( H\right) =\left\{ w_{i}\left| i\in \left[ 1,\left\lfloor
t/2\right\rfloor +1\right] \right. \right\} $ and the edge set 
$E\left( H\right) =\left\{ w_{i}w_{j}\left| i\in \left[ 1,\left\lfloor
t/2\right\rfloor \right] \text{ and }i<j\leq \left\lfloor t/2\right\rfloor
+1\right. \right\}$.
Thus, $\left| V\left( H\right) \right| =\left\lfloor t/2\right\rfloor +1$ and%
\begin{equation*}
\left| E\left( H\right) \right| =\left\lfloor t/2\right\rfloor +\left(
\left\lfloor t/2\right\rfloor -1\right) +\cdots +1=\binom{\left\lfloor
t/2\right\rfloor +1}{2}\text{.}
\end{equation*}%
Consequently, $H=K_{\left\lfloor t/2\right\rfloor +1}$ so that $K_{\left\lfloor
t/2\right\rfloor +1}\subseteq F_{t}$. Since $K_{1,s-1}\subseteq F_{s}$ and $%
K_{\left\lfloor t/2\right\rfloor +1}\subseteq F_{t}$, it follows from
Theorem \ref{Chavatal} that 
\begin{equation*}
r\left( F_{s},F_{t}\right) \geq r\left( K_{1,s-1},K_{\left\lfloor
t/2\right\rfloor +1}\right) =1+\left( s-1\right) \left\lfloor
t/2\right\rfloor\text{.}
\end{equation*}
\end{proof}

Let $t$ be an integer with $t \geq 2$. Then the empty graph $G=\left( t-1\right) K_{1}$ does not contain $F_{2}=K_{2}$
as a subgraph since $G$ does not have edges. Also, the complete graph $\overline{G}%
=K_{t-1}$ does not contain $F_{t}$ as a subgraph since $K_{t-1}$ has order $t-1$, while $F_{t}$ has order $t$.
Consequently, $r\left( F_{2},F_{t}\right) \geq t$. On the other hand, it is straightforward to see that $r\left(2,t\right)=t$, which implies that $r\left(F_{2}, F_{t}\right) \leq r\left(2,t\right)=t$. Combining the two inequalities, we have the following.

\begin{proposition}
\label{r(F2,Ft)} For every integer $t \geq 2$,
\begin{equation*}
r\left( F_{2},F_{t}\right)=t\text{.}
\end{equation*}
\label{r(F2,Ft)}
\end{proposition}

A formula for $r\left( P_{3}, G\right) $, where $G$ is a graph of order $n \geq 2$ without isolated vertices, was derived by Chv\'{a}tal
and Harary \cite{Chvatal-HararyIII}. It involves the concepts of $1$-factor
and edge independence number $\beta _{1}\left( \overline{G}\right) $ of the complement of a graph $G$ as stated next.

\begin{theorem}
\label{Chvatal-Harary}
\label{ramsey_F3_G}For every graph $G$ of order $n\geq 2$ without isolated
vertices, 
\begin{equation*}
r\left( P_{3},G\right) =\left\{ 
\begin{tabular}{ll}
$n$ & if $\overline{G}$ has a $1$-factor \\ 
$2n-2\beta _{1} \left(\overline{G}\right) -1$ & otherwise.%
\end{tabular}%
\right.
\end{equation*}
\end{theorem}

It is immediate that $\overline{F}_{t}=F_{t-1} \cup K_{1}$ and $\beta_{1}\left(\overline{F}_{t}\right) =\left\lfloor \left(t-1\right)/2\right\rfloor$ for integers $t \geq 2$. Therefore, the Ramsey number $r\left( F_{3},F_{t}\right) $ is determined by letting $G=F_{t}$ in Theorem \ref{ramsey_F3_G} as indicated next.

\begin{corollary} For every integer $t \geq 2$,
\label{r(F3,Ft)}
\begin{equation*}
r\left( F_{3},F_{t}\right) =\left\{ 
\begin{tabular}{ll}
$t+1$ & if $t$ is even\\ 
$t$ & if $t$ is odd.%
\end{tabular}%
\right.
\end{equation*}
\end{corollary}

The preceding result shows that the bound presented in Theorem \ref{ramsey_lower} is exact when $s=3$ and $t \geq 2$.

There is an equivalent formulation of the definition of $r\left(G_{1}, G_{2}\right)$ in terms of $2$-colorings of a complete graph. Namely, $r\left(G_{1},G_{2}\right)$ is the least positive integer $n$ such that for every graph $G$ of order $n$, there is either a subgraph isomorphic to $G_{1}$, all whose edges are colored red (a \emph{red} $G_{1}$) or a subgraph isomorphic to $G_{2}$, all whose edges are colored blue (a \emph{blue} $G_{2}$). Using this definition, we next determine the Ramsey number $r\left(F_{4}, F_{t}\right)$.

\begin{theorem} For every integer $t \geq 3$,
\begin{equation*}
r\left(F_{4},F_{t}\right)=2t-1\text{.}
\end{equation*}
\label{r(F4,Ft)}
\end{theorem}
\begin{proof}
It is known from Corollary \ref{r(F3,Ft)} that  
\begin{equation*}
r\left(F_{4},F_{3}\right)=r\left(F_{3},F_{4}\right)=4+1=5=2\cdot 3-1\text{,}
\end{equation*}
which implies that the result is true for $t=3$. Thus, we assume that $t$ is an integer with $t \geq 4$. First, we show that $r\left(F_{4},F_{t}\right) \geq 2t-1$. The complete bipartite graph $G$=$K_{t-1,t-1}$ does not contain $K_{3}$ as a subgraph. However, $F_{4}$ contains $K_{3}$ as a subgraph, and hence $F_{4}$ is not a subgraph of $G$. On the other hand, the graph $\overline{G}=2K_{t-1}$ does not contain $F_{t}$ as a subgraph since each component of $\overline{G}$ has order $t-1$ and $F_{t}$ has order $t$. Therefore, $r\left(F_{4},F_{t}\right) \geq 2t-1$.

Next, we show that $r\left(F_{4},F_{t}\right) \leq 2t-1$. Consider any $2$-coloring of the edges of $K_{2t-1}$, and let $v$ be a vertex of $K_{2t-1}$. We show that there is either a red $F_{4}$ or a blue $F_{t}$. Since $v$ is incident with $2t-2$ edges, it follows from the pigeonhole principle that at least $t-1$ of these $2t-2$ edges are colored the same, say red. Without loss of generality, we assume that the edges $vv_{i}$ ($i \in \left[1,t-1\right]$) in $K_{2t-1}$ are colored red. If any one of the edges $v_{i}v_{j}$ ($1 \leq i<j \leq t-1$) is colored red, then there is a red $F_{4}$; otherwise, all $\left(t-1\right)\left(t-2\right)/2$ of these edges are blue, producing a blue $K_{t-1}$. Now, consider $2$-colorings of the remaining edges $vw_{j}$ ($j \in \left[1,t-1\right]$), $v_{i}w_{j}$ ($i,j \in \left[1,t-1\right]$) and $w_{i}w_{j}$ ($1 \leq i<j \leq t-1$) of $K_{2t-1}$. If any of the edges $v_{i}w_{j}$ ($i,j \in \left[1,t-1\right]$) is colored blue, then the resulting $2$-coloring has a blue $K_{t}$, since the $2$-coloring has a blue graph induced by the set $\left\{ v_{i}\left\vert i\in \left[ 1,t-1\right] \right. \right\}$ and this blue graph is isomorphic to $K_{t-1}$. Consequently, this blue subgraph is isomorphic to $K_{t}$ and contains a blue $F_{t}$. Hence, all edges $v_{i}w_{j}$ ($i,j \in \left[1,t-1\right]$) are colored red. This $2$-coloring produces a red subgraph induced by the set $\left\{ v_{i}\left\vert i\in \left[ 1,t-1\right]  \right. \right\}$ and this red subgraph is isomorphic to $K_{t-1,t}$. Then if we introduce a new red edge among the vertices $v,w_{1},w_{2},\ldots,w_{t-1}$, we would produce a red $F_{4}$. Thus, all vertices $v,w_{1},w_{2},\ldots,w_{t-1}$ must be incident with blue edges forming a blue $K_{t}$ and resulting in a blue $F_{t}$ as a subgraph. Therefore, every $2$-coloring of the edges of $K_{2t-1}$ produces either a red $F_{4}$ or a blue $F_{t}$ so that $r\left(F_{4},F_{t}\right) \leq 2t-1$.  Combining the two inequalities, we have $r\left(F_{4},F_{t}\right) = 2t-1$.
\end{proof}

The exact values of $r\left( F_{s}, F_{t}\right) $ given in Table 1 are obtained from Proposition \ref{r(F2,Ft)}, Corollary \ref{r(F3,Ft)}, and Theorem \ref{r(F4,Ft)}, and indicate that the bound presented in Theorem \ref{ramsey_lower} is exact when $2\leq s\leq t\leq 4$. The Ramsey numbers $r\left( F_{s}, F_{t}\right) $ for $2 \leq s \leq t \leq 4 $ can also be found in papers \cite{Chvatal-HararyII} and \cite{Chvatal-HararyIII} by Chv\'{a}tal and Harary. 

\begin{table}[ht]
\caption{Small Ramsey numbers $r\left( F_{s},F_{t}\right) $ }
\begin{center}
\begin{tabular}{ccccc}
\\[-0.9em]
\hline
\\[-0.9em]
$G_{1}$ & $\quad \quad \quad \quad $ & $\ \ \ G_{2}\quad $ & $\quad \quad
\quad $ & $r\left( G_{1},G_{2}\right) $ \\ 
\\[-0.9em]
\hline
\\[-0.9em]
$F_{2}$ & $\quad $ & $F_{2}$ & $\quad $ & $2$ \\ 
$F_{2}$ & $\quad $ & $F_{3}$ & $\quad $ & $3$ \\ 
$F_{2}$ & $\quad $ & $F_{4}$ & $\quad $ & $4$ \\ 
$F_{3}$ & $\quad $ & $F_{3}$ & $\quad $ & $3$ \\ 
$F_{3}$ & $\quad $ & $F_{4}$ & $\quad $ & $5$ \\ 
$F_{3}$ & $\quad $ & $F_{5}$ & $\quad $ & $5$ \\
$F_{4}$ & $\quad $ & $F_{4}$ & $\quad $ & $7$ \\ 
$F_{4}$ & $\quad $ & $F_{5}$ & $\quad $ & $9$ \\ 
$F_{4}$ & $\quad $ & $F_{6}$ & $\quad $ & $11$ \\ 
$F_{4}$ & $\quad $ & $F_{7}$ & $\quad $ & $13$ \\ 
\\[-0.9em]
\hline
\end{tabular}%
\end{center}
\end{table}

The upper bound is unknown for $r\left(F_{s}, F_{t}\right)$ when $5 \leq s \leq t$. This motivates us to propose the next problem.

\begin{problem}
Find a good upper bound for $r\left( F_{s},F_{t}\right) $ when $5 \leq s \leq t$.
\end{problem}

The only known Ramsey numbers $r\left(F_{s}, F_{t}\right)$ for $4 \leq s  \leq t \leq 5$ are $r\left(F_{4}, F_{4}\right)=7$
and $r\left(F_{4}, F_{5}\right)=9$ (see Table 1), and $r\left(F_{5}, F_{5}\right)=10$ (see \cite{Burr}). Thus, it is natural
to propose the next problem.

\begin{problem}
Determine the exact values of $r\left( F_{s},F_{t}\right) $ for any integers $s \geq 5$ and $t \geq 6$.
\end{problem}

\section{Lower bounds and small values for $f\left( n\right) $}
Define $f\left( n\right) =\max\left\{\mathrm{str}\left( G\right) +%
\mathrm{str}\left( \overline{G}\right)\right\} $, where the maximum is taken over all graphs $G$ of order $n$ such that neither $G$ nor $\overline{G}$ is an empty graph. 
In this section, we show a sharp lower bound for $f\left( n\right) $. With the knowledge from the preceding section, 
we also provide another lower bound for $f\left( n\right) $.
To proceed with these, it is important to notice that there exist no integers $s$ and $t$ such that $r\left( F_{s}, F_{t}\right) >n \geq \max \left\{ s,t\right\} $ for $n=3$ (see Table 1). Thus, we assume that $n\geq 4$ when we consider the number $f\left( n\right) $ together with the Ramsey number $r\left( F_{s}, F_{t}\right)$.

We are now prepared to present lower bounds for $f\left(n\right)$. We begin with a lemma.

\begin{lemma}
\label{basic_1} If $r\left( F_{s},F_{t}\right) >n \geq \max \left\{
s,t\right\} $, then 
\begin{equation*}
f\left( n\right) \geq 4n-\left(s+t\right)+2\text{,}
\end{equation*}
where $s,t \in \left[2,n-1\right]$.
\end{lemma}

\begin{proof}
By assumption, there exists a graph $G$ of order $n$ such that $%
F_{s}\nsubseteqq G$ and $F_{t}\nsubseteqq \overline{G}$. For every integer $%
n \geq \max \left\{ s,t\right\} $, where  $s,t \in \left[2,n-1\right]$, Theorem \ref{main2} and the latest
statement yield 
$\mathrm{str}\left( G\right) \geq 2n-s+1$ and $\mathrm{str}\left( 
\overline{G}\right) \geq 2n-t+1$.
Hence, $f\left( n\right) \geq 4n-\left(s+t\right)+2$.
\end{proof}

It is now possible to present the following lower bound.

\begin{theorem}
\label{lower_bound1} For every integer $n\geq 4$,
\begin{equation*}
f\left( n\right) \geq 3n+\left\lfloor n/2\right\rfloor -3\text{.}
\end{equation*}
\end{theorem}

\begin{proof}
Let $n$ be an integer with $n \geq 4$, and consider the graph $H=K_{\left\lfloor
n/2\right\rfloor }\cup K_{\left\lceil n/2\right\rceil }$ and its complement $\overline{H}=K_{\left\lfloor n/2\right\rfloor ,\left\lceil n/2\right\rceil }$. 
It is known from \cite{IOT} that 
$\mathrm{str}\left( K_{s}\cup K_{t}\right) =2\left( s+t\right) -3$
for every two integers $s\geq 2$ and $t\geq 2$, 
which implies that 
\begin{equation*}
\mathrm{str}\left( H\right)=2\left(\left\lfloor n/2\right\rfloor
+\left\lceil n/2\right\rceil\right)-3=2n-3\text{.}
\end{equation*}
It is also
known from \cite{IMO1} that $\mathrm{str}\left( K_{s,t}\right) =2s+t$
for every two integers $s$ and $t$ with $1\leq s\leq t$, which implies that 
\begin{equation*}
\mathrm{str}\left( \overline{H}\right) =2\left\lfloor n/2\right\rfloor
+\left\lceil n/2\right\rceil =n+\left\lfloor n/2\right\rfloor\text{.}
\end{equation*}
Therefore, $\mathrm{str}\left( H\right) +\mathrm{str}\left( \overline{H}\right) =3n+\left\lfloor n/2\right\rfloor -3$.
This implies that 
\begin{equation*}
\begin{split}
f\left( n\right)&=\max\left\{\mathrm{str}\left( G\right) +\mathrm{str}\left( \overline{G}\right) \right\}\geq \mathrm{str}\left( H\right) +\mathrm{str}\left( \overline{H}\right) \\
&=3n+\left\lfloor n/2\right\rfloor -3\text{,}
\end{split}
\end{equation*}
where the maximum is taken over all graphs $G$ of order $n$ such that neither $G$ nor $\overline{G}$ is an empty graph. 
\end{proof}

There is exactly one pair of nonempty graphs $G$ and $\overline{G}$ of order 
$3$, namely, $G=K_{1,2}$ and $\overline{G}=K_{1}\cup K_{2}$. With the aid of Observation \ref%
{observation} and Lemma \ref{lemma_trivial}, it is easy to determine that $\mathrm{str}\left( G\right) =4$ and $\mathrm{str%
}\left( \overline{G}\right) =3$, which implies that $f\left( 3\right) =7$. 
This indicates that $f\left( n\right)$ attains the bound given in Theorem \ref{lower_bound1} for $n$ = $3$.

We now present the following result, which gives a potentially improved
lower bound for $f\left( n\right) $.

\begin{theorem}
\label{lower-bound2} For every integer $n\geq 4$,
\begin{equation*}
f\left( n\right) \geq 4n-2\left\lceil \left(3+\sqrt{8n-7}\right)/2\right\rceil +2\text{.}
\end{equation*}
\end{theorem}

\begin{proof}
If we let $s=\left\lceil\left( 3+\sqrt{8n-7}\right) /2\right\rceil$\ and $n\geq 4$, 
then $s\geq 4$. It follows from Theorem \ref{ramsey_lower} that
\begin{equation*}
r\left( F_{s},F_{s}\right) \geq 1+\left( s-1\right) \left\lfloor
s/2\right\rfloor > n\text{.}
\end{equation*} 
On the other hand, the inequality $n \geq s=\left\lceil\left( 3+\sqrt{8n-7}\right) /2\right\rceil$ holds for every integer $n\geq 4$. Thus, we
conclude by Lemma \ref{basic_1} that 
\begin{equation*}
f\left( n\right) \geq 4n-2\left\lceil \left(3+\sqrt{8n-7}\right)/2\right\rceil +2\text{.}
\end{equation*}
\end{proof}

%\begin{problem}
%Determine whether there are graphs $G$ of order $n\geq 27$ for which 
%\begin{equation*}
%\mathrm{str}\left( G\right) +\mathrm{str}\left( \overline{G}\right)
%=4n-2\left\lceil \left(3+\sqrt{8n-7}\right)/2\right\rceil +2\text{.}
%\end{equation*}
%\end{problem}

It is now important to observe that Lemma \ref{basic_1} implies that 
\begin{equation*}
f\left( n\right) = 4n-\min \left\{ s+t\left| r\left( F_{s},F_{t}\right)>n\right. \right\} +2
\end{equation*}
if $n \geq \max \left\{ s,t\right\}$ for a pair of integers $s$ and $t$, where the minimum occurs.
Table 2 summarizes the small values for $f (n)$ obtained from the above observations and the facts from the preceding section.

\begin{table}[ht]
\caption{Small values for $f\left(n\right)$}
\begin{tabular}{lll}
\hline
\\[-0.9em]
$n$ & \multicolumn{1}{c}{$f\left(n\right)$} & \multicolumn{1}{c}{Reasons for equality}  \\ \\[-0.9em] 
\hline 
\\[-0.9em]
$3$ & $\quad 7 \quad $ & $ \quad\quad G=K_{1,2}$ and $\overline{G}=K_{1} \cup K_{2}$\\ 
$4$ & $\quad 11 \quad $ & $ \quad\quad r\left( F_{3},F_{4}\right)=5$\\ 
$5 $ & $\quad 14 $ & $\quad\quad r\left( F_{4},F_{4}\right)=7$\\ 
$6$ & $\quad 18 $ & $\quad\quad r\left( F_{4},F_{4}\right)=7$ \\ 
$7$ & $\quad 21 $ & $\quad\quad r\left( F_{4},F_{5}\right)=9$ \\ 
$8$ & $\quad 25 $ & $\quad\quad r\left( F_{4},F_{5}\right)=9$ \\ 
$9$ & $\quad 28$ & $\quad\quad r\left( F_{5},F_{5}\right)=10$ \\ 
$10$ & $\quad 32$ & $\quad\quad r\left( F_{4},F_{6}\right)=11$ \\ 
$11$ & $\quad 35$ & $\quad\quad r\left( F_{4},F_{7}\right)=13$ and  $r\left( F_{5},F_{6}\right) \geq 13$\\ 
$12$ & $\quad 39$ & $\quad\quad r\left( F_{4},F_{7}\right)=13$ and  $r\left( F_{5},F_{6}\right) \geq 13$\\ 
\\[-0.9em]
\hline
\end{tabular}%
\end{table}

As mentioned above, $f\left(n\right)$ attains the bound presented in Theorem 3.1 for $n=3$. Indeed, $f\left(n\right)$ attains the same bound for $n \in \left[4,12 \right]$.
However, we do not know whether the case for $n \geq 13$. Thus, we propose the next two problems.

\begin{problem}
Find good lower and upper bounds for $f\left( n\right) $ when $n \geq 13$.
\end{problem}

\begin{problem}
Determine the exact values of $f\left(n\right)$ for integers $n \geq 13$.
\end{problem}

\section {Additional notes}

In this section, we extend Theorems \ref{main1}, \ref{main2}, and \ref{main3} stated in the introduction. To achieve this, we introduce the first element in the class of graphs $F_{k}$ defined in the introduction as $F_{1}=K_{1}$. We also investigate bounds for $\mathrm{str}\left( G\right) +\mathrm{str}\left( \overline{G}\right)$ whenever $G$ and $\overline{G}$ are nonempty graphs of order $n$.

Theorem \ref{main1} can be extended as indicated next.
\begin{theorem}
\label{extension1} Let $G$ be a nonempty graph of order $n$. Then $\mathrm{str}\left(G\right) \leq 2n-k$ if and only if $\overline{G}$ contains $F_{k}$ as a subgraph, where $k\in \left[ 1,n\right] $.
\end{theorem}
\begin{proof}
It is known from Theorem \ref{main1} that the theorem is true for all values of $k$ with $k \in \left[2,n-1\right]$. 

For $k=1$, the result is true since the maximum possible strength is $2n-1$. Thus, assume that $k=n \geq 2$. Then $F_{n} \subseteq \overline{G}$. This implies that either $G=F_{n-1} \cup K_{1}$ or $G=H \subseteq F_{n-1} \cup K_{1}$. 
Let $F_{n-1} \cup K_{1}$ be the graph with 
\begin{equation*}
V\left( F_{n-1} \cup K_{1}\right)=\left\{u \right\} \cup \left\{ v_{i}\left\vert i\in \left[ 1,n-1\right] \right. \right\} 
\end{equation*}
and 
\begin{equation*}
E\left( F_{n-1} \cup K_{1}\right) =\left\{ v_{i}v_{j}\left\vert i\in \left[
1,\left\lfloor \left(n-1\right)/2\right\rfloor \right] \text{ and }j\in \left[ 1+i,n-i%
\right] \right. \right\}\text{,}
\end{equation*}
and consider the numbering $f:V\left( F_{n-1} \cup K_{1}\right) \rightarrow \left[ 1,n\right] $ such that $f\left(u\right)=n$ and $f\left(v_i\right)=i$ ($i \in \left[ 1,n-1\right]$). 
Then $f$ has the property that 
\begin{equation*}
\mathrm{str}\left( F_{n-1} \cup K_{1}\right) \leq f\left(v_{1}\right)+f\left(v_{n-1}\right)=1+\left(n-1\right)=n\text{.}
\end{equation*}
This together with $H  \subseteq F_{n-1} \cup K_{1}$ implies that 
$\mathrm{str}\left(H\right) \leq \mathrm{str}\left(F_{n-1} \cup K_{1}\right)\leq n$,
giving the desired result.
\end{proof}

In light of Theorem \ref{extension1}, Theorems \ref{main2} and \ref{main3} are now extended as follows.
\begin{theorem}
\label{extension2} Let $G$ be a nonempty graph of order $n$. Then $%
\mathrm{str}\left( G\right) \geq 2n-k+1$ if and only if $\overline{G}$ does
not contain $F_{k}$ as a subgraph, where $k\in \left[ 1,n\right] $.
\end{theorem}

\begin{theorem}
\label{extension3} Let $G$ be a nonempty graph of order $n$ with $\delta \left( G\right)
=n-k$, where $k\in \left[ 1,n\right] $. Then $\mathrm{str}\left( G\right)
=2n-k$ if and only if $\overline{G}$ contains $F_{k}$ as a subgraph.
\end{theorem}

In 1956, Nordhaus and Gaddum \cite{NG} established lower and upper bounds for the sum and the product of the chromatic number of a graph and its complement. Such inequalities established for any parameter have become known as \emph{Nordhaus-Gaddum inequalities}. Relations of a similar type have been found in several hundred papers on many other parameters. Readers interested in further knowledge on this topic may consult the survey paper by Aouchiche and Hansen \cite{AH1}. 

In 1974, Chartrand and Schuster \cite{GS} proved the following lower and upper bounds for the sum of the independence number $\beta\left(G\right)$ of a graph $G$ and its complement $\overline{G}$. Their lower bound is related to Ramsey numbers $r\left(s,t\right)$.
\begin{theorem}
\label{Nordhaus–Gaddum bound}
For every graph $G$ of order $n$, 
\begin{equation*}
 \sigma_{n} \leq \beta\left( G\right) +\beta\left( \overline{G}\right) \leq n+1\text{,}
\end{equation*}
where $\sigma_{n}=\min \left\{ a+b\left| r\left( a+1,b+1\right)>n\right. \right\}$.
\end{theorem}

The following upper bound for the strength of a graph was found in \cite{IMT}.
\begin{theorem}
\label{upper_bound_strength}
For every graph $G$ of order $n$, 
\begin{equation*}
\mathrm{str}\left( G\right) \leq 2n-\beta\left(G\right)\text{.}
\end{equation*}
\end{theorem}

The lower bounds presented in the preceding section for $f\left(n\right)$ are certainly lower bounds for $\mathrm{str}\left( G\right) +\mathrm{str}\left( \overline{G}\right)$. This together with Theorems \ref{Nordhaus–Gaddum bound} and \ref{upper_bound_strength} provides the following result.

\begin{corollary}
\label{lower_upper_bounds}
If $G$ and $\overline{G}$ are nonempty graphs of order $n$, then
\begin{equation*}
\max\left\{\rho_{n},\rho^{\prime}_{n}\right\} \leq \mathrm{str}\left( G\right) +\mathrm{str}\left( \overline{G}\right) \leq 4n-\sigma_{n}\text{,}
\end{equation*}
where $\rho_{n}=3n+\left\lfloor n/2\right\rfloor -3$ and $\rho^{\prime}_{n}=4n-2\left\lceil \left(3+\sqrt{8n-7}\right)/2\right\rceil +2$. 
\end{corollary}

Next, we consider the computational results on the bounds for the sum of the strength of a graph and its complement. The only known Ramsey numbers $r\left(s,t\right) $ for $3 \leq s \leq t \leq 5$ (see \cite{Radziszowsiki}) are 
\begin{equation*}
\begin{tabular}{lll}
$r\left(3,3\right)=6$ & $r\left(3,6\right)=18$ & $r\left(3,9\right)=36$  \\
$r\left(3,4\right)=9$ & $r\left(3,7\right)=23$ & $r\left(4,4\right)=18$  \\
$r\left(3,5\right)=14$ & $r\left(3,8\right)=28$ & $r\left(4,5\right)=25$. 
\end{tabular} 
\end{equation*}
Using this information together with observations so far, we can determine the values of $\sigma_{n}$ and the bounds for $\mathrm{str}\left( G\right) +\mathrm{str}\left( \overline{G}\right)$ given in Corollary \ref{lower_upper_bounds} when $n \in \left[3,35\right]$. Tables 3 and 4 list the values of $\sigma_{n}$ and the values of $\rho_{n}$, $\rho^{\prime}_{n}$ and $4n-\sigma_{n}$ for $n \in \left[3,35\right]$, respectively. Chartrand and Schuster \cite{GS} have computed the values of $\sigma_{n}$ for $n \in \left[1,25\right]$. 

\begin{table}[ht]
\caption{Values of $\sigma_{n}$ for $n \in \left[3,35\right]$ }
\begin{tabular}{rrl}
\\[-0.9em]
\hline
\\[-0.9em]
\multicolumn{1}{c}{$\quad n$} & \multicolumn{1}{c}{$\quad \sigma_{n}$} & \text{Reasons for equality} \\ 
\\[-0.9em]
\hline
\\[-0.9em]
$\quad\left[3,5\right]$ & $4$ & $\quad\quad r\left(3,3\right)=6$\\
$\quad\left[6,8\right]$ & $5$ & $\quad\quad r\left(3,4\right)=9$\\ 
$\quad\left[9,17\right]$ & $6$ & $\quad\quad r\left(4,4\right)=18$\\ 
$\quad\left[18,24\right]$ & $7$ & $\quad\quad r\left(4,5\right)=25$\\ 
$\quad\left[25,27\right]$ & $9$ & $\quad\quad r\left(3,8\right)=28$\\ 
$\quad\left[28,35\right]$ & $10$ & $\quad\quad r\left(3,9\right)=36$\\ 
\\[-0.9em]
\hline

\end{tabular}%
\end{table}

\begin{table}[ht]
\caption{Bounds for $\mathrm{str}\left( G\right) +\mathrm{str}\left( \overline{G}\right)$ }
\begin{center}
\begin{tabular}{rrrr}
\hline
\\[-0.9em]
$n$ & $\quad\rho_{n}$ & $\quad\rho^{\prime}_{n}$ & $\quad4n-\sigma_{n}$ \\ 
\\[-0.9em]
\hline
\\[-0.9em]
$3$ & $\quad\quad 7$ & $\quad\quad 6$& $\quad 8$\\ 
$4$ & $\quad\quad 11$ & $\quad\quad 10$& $\quad 12$\\ 
$5$ & $\quad\quad 14$ & $\quad\quad 12$& $\quad 16$\\ 

$6$ & $\quad\quad 18$ & $\quad\quad 16$ & $\quad 19$\\ 
$7$ & $\quad\quad 21$ & $\quad\quad 20$ & $\quad 23$\\ 
$8$ & $\quad\quad 25$ & $\quad\quad 22$ & $\quad 27$\\ 

$9$ & $\quad\quad 28$ & $\quad\quad 26$ & $\quad 30$ \\ 
$10$ & $\quad\quad 32$ & $\quad\quad 30$& $\quad 34$\\
$11$ & $\quad\quad 35$ & $\quad\quad 34$& $\quad 38$\\ 
$12$ & $\quad\quad 39$ & $\quad\quad 36$& $\quad 42$\\ 
$13$ & $\quad\quad 42$ & $\quad\quad 40$ & $\quad 46$\\ 
$14$ & $\quad\quad 46$ & $\quad\quad 44$ & $\quad 50$\\ 
$15$ & $\quad\quad 49$ & $\quad\quad 48$ & $\quad 54$\\ 
$16$ & $\quad\quad 53$ & $\quad\quad 52$ & $\quad 58$ \\
$17$ & $\quad\quad 56$ & $\quad\quad 54$ & $\quad 62$ \\

$18$ & $\quad\quad 60$ & $\quad\quad 58$& $\quad 65$\\
$19$ & $\quad\quad 63$ & $\quad\quad 62$& $\quad 69$\\ 
$20$ & $\quad\quad 67$ & $\quad\quad 66$& $\quad 73$\\ 
$21$ & $\quad\quad 70$ & $\quad\quad 70$ & $\quad 77$\\ 
$22$ & $\quad\quad 74$ & $\quad\quad 74$ & $\quad 81$\\ 
$23$ & $\quad\quad 77$ & $\quad\quad 76$ & $\quad 85$\\ 
$24$ & $\quad\quad 81$ & $\quad\quad 80$ & $\quad 89$ \\ 

$25$ & $\quad\quad 84$ & $\quad\quad 84$& $\quad 91$\\
$26$ & $\quad\quad 88$ & $\quad\quad 88$& $\quad 95$\\ 
$27$ & $\quad\quad 91$ & $\quad\quad 92$& $\quad 99$\\ 

$28$ & $\quad\quad 95$ & $\quad\quad 96$& $\quad 102$\\
$29$ & $\quad\quad 98$ & $\quad\quad 100$& $\quad 106$\\ 
$30$ & $\quad\quad 102$ & $\quad\quad 102$& $\quad 110$\\ 
$31$ & $\quad\quad 105$ & $\quad\quad 106$ & $\quad 114$\\ 
$32$ & $\quad\quad 109$ & $\quad\quad 110$ & $\quad 118$\\ 
$33$ & $\quad\quad 112$ & $\quad\quad 114$ & $\quad 122$\\ 
$34$ & $\quad\quad 116$ & $\quad\quad 118$ & $\quad 126$ \\
$35$ & $\quad\quad 119$ & $\quad\quad 122$ & $\quad 130$ \\
\\[-0.9em]
\hline
\end{tabular}
\end{center}
\end{table}

As mentioned above, the known Ramsey numbers $r\left(s,t\right)$ are very limited. Through this, the preceding result is not ideal in the sense that it depends on $r\left(s,t\right)$. Therefore, efficiently computable upper bounds are necessary to establish. This motivates us to propose the next problem.

\begin{problem}
Find an efficiently computable upper bound for $\mathrm{str}\left( G\right) +\mathrm{str}\left( \overline{G}\right)$ when $G$ and $\overline{G}$ are nonempty graphs of order $n$.
\end{problem}

\section{Conclusions}
In this paper, we have established a lower bound for $r\left(F_{s}, F_{t}\right)$ (see Theorem \ref{ramsey_lower}).  We also have provided the exact values of $r\left(F_{s},F_{t}\right)$ for $s \in \left[2,4\right]$ (see Proposition \ref{r(F2,Ft)}, Corollary \ref{r(F3,Ft)}, and Theorem \ref{r(F4,Ft)}). From these results, we know that the bound given in Theorem \ref{ramsey_lower} is exact when $s=3$ and $t \geq 2$ and have rediscovered the known Ramsey numbers $r\left(F_{s}, F_{t}\right)$ for $2 \leq s \leq t \leq 4$ (see Table 1). In addition to these, we have presented two lower bounds for $f\left(n\right)$ (see Theorems \ref{lower_bound1} and \ref{lower-bound2}) and have determined the exact values of $f\left(n\right)$ for $n \in \left[3,12 \right]$ (see Table 2). Furthermore, we have extended the known necessary and sufficient conditions involving the strength of a graph (see Theorems \ref{extension1}, \ref{extension2}, and \ref{extension3}) and have found lower and upper bounds for $\mathrm{str}\left( G\right) +\mathrm{str}\left( \overline{G}\right)$ (see Corollary \ref{lower_upper_bounds}). Throughout this paper, we have proposed some open problems arising from our study. Finally, with this paper, the authors hope that interest in the strength of graphs will be aroused among those who study the theory of Ramsey numbers or graph labelings.

\subsubsection*{$\emph{Acknowledgment}$}
The authors dedicate this paper to Susana-Clara L\'{o}pez Maship. 
Her inspiration and dedication have brought new avenues into graph labeling and other related topics.

\end{document}